\title{ \LARGE \bf%
$L^p$-asymptotic stability analysis of a 1D wave equation with a nonlinear damping
}

\author{Yacine Chitour$^{1}$, Swann Marx$^{2}$ and Christophe Prieur$^{3}$}

\documentclass[12pt]{article}
\usepackage{amsfonts,latexsym,amsmath,amssymb,amsthm}
\usepackage[mathscr]{eucal}
\usepackage{graphicx,color}
\usepackage{comment}
\usepackage{hyperref}
\usepackage{epstopdf}
\usepackage[margin = 1in]{geometry}
\newtheorem{theorem}{Theorem}

\newtheorem{proposition}{Proposition}

\newtheorem{definition}{Definition}

\newtheorem{example}{Example}

\begin{document}
\footnotetext[1]{Laboratoire des Signaux et Systemes (L2S), CNRS - CentraleSupelec - Universite Paris-Sud, 3 rue Joliot Curie, F-91192, Gif-sur-Yvette, France}
\footnotetext[2]{Universit\'e de Toulouse - CNRS - 7 avenue du colonel Roche, F-31400 Toulouse, France}
\footnotetext[3]{Univ. Grenoble Alpes, CNRS, Grenoble INP, Gipsa-lab,
F-38000 Grenoble, France}
\maketitle
\abstract{This paper is concerned with the asymptotic stability analysis of a one dimensional wave equation with Dirichlet boundary conditions subject to a nonlinear distributed damping with an $L^p$ functional framework, $p\in [2,\infty]$. Some well-posedness results are provided together with exponential decay to zero of trajectories, with an estimation of the decay rate. The well-posedness results are proved by considering an appropriate functional of the energy in the desired functional spaces introduced by Haraux in \cite{haraux1D}. Asymptotic behavior analysis is based on an attractivity result on a trajectory of an infinite-dimensional linear time-varying system with a special structure, which relies on the introduction of a suitable Lyapunov functional. Note that some of the results of this paper apply for a large class of nonmonotone dampings.}
\tableofcontents

\section{Introduction}

%to be cited: \cite{mcpa2017siam}, \cite{mcp2018ECC}, \cite{map2017mcss}, \cite{seidman2001note}, \cite{slemrod1989mcss}, \cite{brezis2010functional}, \cite{kang2018boundary}, \cite{miyadera1992nl_sg}, \cite{MVC}, \cite{enrike1990exponential}, \cite{haraux1985stabilization}, \cite{amadori2019decay}, \cite{feireisl1993strong}

This paper is concerned with the asymptotic behavior of a one-dimensional wave equation subject to a nonlinear damping with a functional framework relying on $L^p$ spaces with $p\in[2,\infty]$. The one-dimensional wave equation is defined as follows
\begin{equation}
\label{wave-nonmonotone}
\left\{
\begin{split}
&z_{tt}(t,x) =z_{xx}(t,x)-a(x)\sigma(z_t(t,x)),\: (t,x)\in \mathbb{R}_+\times [0,1]\\
&z(t,0)=z(t,1) = 0,\: t\in\mathbb{R}_+\\
&z(0,x) = z_0(x),\: z_t(0,x)=z_1(x),\: x\in [0,1],
\end{split}
\right.
\end{equation}
where $z$ denotes the state belonging to some functional space (to be defined), $a:[0,1]\to \mathbb{R}_+$ is continuous, non zero and bounded by some positive constant $a_\infty$ and $\sigma: \mathbb{R}\to \mathbb{R}$ is a nonlinear function for which we will assume some suitable properties later on. This means in particular that the damping that we are considering might be \emph{localized} (i.e., it acts on a subdomain of $[0,1]$). 

This type of systems may be studied in the context of control of engineering systems. Indeed, the damping term $-a(x)\sigma(z_t)$ corresponds in this case to a \emph{feedback law} and the nonlinearity $\sigma$ models some constraints on the actuator. Here ``damping'' refers to the fact that the natural energy of the system is nonincreasing as a function of the time, which is immediately implied by the property that $\xi\sigma(\xi)\geq 0$, for every real number $\xi$.
A classical example for the nonlinearity $\sigma$ is, for instance, the \emph{saturation}, which imposes amplitude bounds on the control. Linear systems subject to saturations in the actuator have been studied during many years in the context of automatic control theory. The finite dimensional version of \eqref{wave-nonmonotone} can be written as $\frac{d}{dt} z=Az-BB^{\top}\sigma(z)$ with $A$ a stable $n\times n$ matrix (i.e., $A+A^\top\leq 0$), $B$ any $n\times m$ matrix and $\sigma:\mathbb{R}^n\to\mathbb{R}^n$ a saturation mapping so that $z^\top\sigma(z)\geq 0$. Here $m,n$ are positive integers. In
\cite{liu1996finite}, it has been proven that the latter finite-dimensional system is semi-globally exponentially stable, i.e., trajectories converge to zero exponentially with a decay rate depending on the initial condition. Some of the proofs of the present paper are inspired by \cite{liu1996finite}. 

Since decades, under some suitable assumptions on the nonlinearity $\sigma$, asymptotic behavior of \eqref{wave-nonmonotone} has been studied by means of various tools: LaSalle's Invariance Principle (\cite{slemrod1989weak}, \cite{slemrod1989mcss}, \cite{prieur2016wavecone})), multiplier method (\cite{enrike1990exponential}, \cite{alabau1999stabilisation}, \cite{alabau2002indirect}, \cite{alabau2012some}, \cite{MVC}) or Lyapunov functionals (\cite{marx2018stability}). However, while this equation is usually studied in the classical functional space $L^2(0,1)\times H^1_0(0,1)$, our paper is devoted to the case of more general functional spaces, defined as follows
\begin{equation}
\begin{split}
&H_p(0,1):= W_0^{1,p}(0,1)\times L^p(0,1)\\
&D_p(0,1):=W^{2,p}(0,1)\cap W_0^{1,p}(0,1)\times W_0^{1,p}(0,1),
\end{split}
\end{equation}
with $p\in [2,\infty]$. The first functional space is equipped with the following norm:
\begin{equation}
\begin{split}
&\Vert (z,z_t)\Vert_{H_p(0,1)} := \left(\int_0^1 [z_x(x)|^p dx\right)^{\frac{1}{p}} + \left(\int_0^1 |z_t(x)|^p dx\right)^{\frac{1}{p}},\: \forall p\in [2,\infty),\\
&\Vert (z,z_t)\Vert_{H_\infty(0,1)}:=\Vert z_x\Vert_{L^\infty(0,1)} + \Vert z_t\Vert_{L^\infty(0,1)},\: \text{for $p=\infty$.}
\end{split}
\end{equation}
The second one is equipped with the following norm
\begin{equation}
\begin{split}
&\Vert (z,z_t)\Vert_{D_p(0,1)}:=\left(\int_0^1 |z_{xx}(x)|^p dx\right)^{\frac{1}{p}} + \left(\int_0^1 |z_{tx}(x)|^pdx\right)^{\frac{1}{p}},\: \forall p\in [2,\infty)\\
&\Vert (z,z_t)\Vert_{D_\infty(0,1)}:=\Vert z_{xx}\Vert_{L^\infty(0,1)} + \Vert z_{t,x}\Vert_{L^\infty(0,1)},\: \text{for $p=\infty$.}
\end{split}
\end{equation}
To the best of our knowledge, very few is known about the global asymptotic stability of \eqref{wave-nonmonotone} in these functional spaces. This is probably due to the fact that for linear wave equations on domains of $\mathbb{R}^n$, $n\geq 2$, the D'Alembertian operator, given by $ \Box z:=z_{tt}-\Delta z$ (with various boundary conditions), is not a well defined bounded operator in general for a $L^p$ space framework with $p\neq 2$ (see e.g., \cite{peral1980lp}).
In \cite{haraux1D}, it is shown that the D'Alembertian operator in one space dimension (with Dirichlet boundary conditions) is a bounded operator for every $H_p$, $p\in[2,\infty]$, and some additional results on decay rates are obtained for 1D waves with a global damping action (i.e. $a\equiv 1$) and where $\sigma$ is differentiable and satisfying, for some positive constants $k_0,k_1$ and $R\geq r>0$,\footnote{All along the paper, given a function $f:\: x\in\mathbb{R}\mapsto f(x)\in\mathbb{R}$, we will denote by $f^\prime$ its derivative.}
\begin{equation}
k_0|s|^r \leq \sigma^\prime(s)\leq k_1(|s|^r + |s|^R),\:\forall s\in\mathbb{R}.
\end{equation}
In \cite{amadori2019decay}, the authors provide an analysis of the asymptotic behavior of trajectories of \eqref{wave-nonmonotone} assuming that the damping function $a$ is bounded away from zero and the nonlinearity $\sigma$ belongs to a positive sector, i.e., 
\begin{equation}
0<k_1\leq a(x)\leq k_2,\:\forall x\in [0,1],\quad
0<g_1\leq \sigma'(\xi)\leq g_2,\:\forall \xi\in\mathbb{R},\end{equation}  
where $k_1,k_2,g_1,g_2$ are positive constants.
The authors propose some error estimation of approximated solutions (i.e., discretized ones) and establish exponential convergence in $L^\infty$ towards zero, provided that the initial conditions are of bounded variations. The techniques used in this paper are based on the theory of scalar conservation laws, which makes indeed the analysis in those spaces really natural. In constrast with these papers, our results apply for more general nonlinearities, with, in some cases, less regular initial conditions, and we even treat the case where the damping action is localized, i.e., it acts on a subdomain of $[0,1]$, which does not hold true for \cite{haraux1D}, neither for \cite{amadori2019decay}.

Note that, assuming that the initial conditions are in $H_\infty$, we provide some exponential stability results with decay rate estimate in the case where $\sigma$ is not monotone, which is a non trivial task in general. Indeed, as illustrated in many papers (\cite{alabau2002indirect}, \cite{alabau2012some}, \cite{alabau1999stabilisation}, \cite{enrike1990exponential}, \cite{haraux1D}, \cite{marx2018stability}, \cite{slemrod1989mcss}, \cite{seidman2001note}, etc.), the monotone property of the nonlinearity is crucial to firstly prove the asymptotic stability of the system under consideration, by applying the LaSalle's Invariance principle, and secondly estimate the decay rates. There exists a vast litterature about linear PDEs subject to monotone nonlinear dampings. For instance, in \cite{slemrod1989mcss} and \cite{seidman2001note}, asymptotic stability of the origin of abstract control systems subject to monotone nonlinear dampings is proved, using an infinite-dimensional version of LaSalle's invariance principle. These results have been then extended to more general infinite-dimensional systems in \cite{map2017mcss}. More recently, in \cite{marx2018stability}, some decay rate estimates have been provided for such systems via Lyapunov techniques. Let us also mention \cite{prieur2016wavecone} and \cite{mcpa2017siam}, where a wave equation and a nonlinear Korteweg-de Vries, respectively, subject to a nonlinear monotone damping are considered and where the global asymptotic stability is proved. 

Let us note however that some results have been obtained for the case where $\sigma$ is not monotone. In \cite{slemrod1989weak}, a weak asymptotic convergence result decay is obtained thanks to a relaxed LaSalle's invariance principle. Using some compensated compactness techniques, the convergence of the trajectories of a one-dimensional wave equation with nonmonotone damping is proved in \cite{feireisl1993strong}. Note however that the estimation of the decay rate is not provided. Finally, exponential stability with decay rate estimates for such an equation are given in \cite{MVC}, but with a particular nonlinearity, that is less general than the one we consider in our paper. Note however that the results of this paper provide some nice exponential decay results for damped wave equations in a domain whose dimension is larger than one.

In our paper, after introducing a general nonlinear nonmonotone damping, we propose some well-posedness results of a one-dimensional wave equation subject to such a damping. Furthermore, using Lyapunov techniques for linear time-varying systems collected in Section \ref{sec_appendix}, we  derive exponential stability with estimates of the decay rate of the trajectory of this system in $H_p$ with $p\in [2,\infty)$ provided that the initial conditions are in $H_\infty$, and without assuming that $\sigma$ is monotone. On the other hand, assuming that $\sigma$ is monotone and that the initial conditions are in $H_p$, with finite $p$, we also get exponential stability with decay rate estimates of the trajectory in $H_q$, where $q$ is such that $2\leq q<p$. Finally, supposing that $\sigma$ is monotone, linearly bounded and continuously differentiable and assuming that the initial conditions belong to $D_\infty$, we give a decay rate estimate of the trajectory in $H_\infty$.

This paper is organized as follows. In Section \ref{sec_main}, the main results of the paper are collected, i.e., we state theorems about well-posedness and exponential stability with decay rate estimates. Section \ref{sec_appendix} introduces a result of independent interest about a specific time-varying infinite-dimensional linear systems, but which is instrumental for the proof of our asymptotic stability theorems. Section \ref{sec_proof} is devoted to the proof of the main results and Section \ref{sec_conclusion} collects some concluding remarks, together with further research lines to be followed. \\%Finally, note that this paper is an extension of \cite{marx2019stability}, which focuses only on the case of nonmonotone dampings.\\

%\textbf{Notation:} For any $p\in [2,\infty)$, the space $L^p(0,1)$ denotes the space of functions $f$ satisfying $\left(\int_0^1 |f(x)|^p dx\right)^{\frac{1}{p}}<+\infty$. The space $L^\infty(0,1)$ denotes the space of functions satisfying $\mathrm{ess}\sup_{x\in [0,1]} |f(x)|\leq +\infty$. For any $p\in [2,\infty]$, the Sobolev space $W^{1,p}(0,1)$ (resp. $W^{2,p}(0,1)$) is defined as follows $W_0^{1,p}(0,1):=\lbrace f\in L^p(0,1)\mid f^\prime\in L^p(0,1)\text{ and }f(0)=f(1)=0\rbrace$ (resp. $W^{2,p}(0,1):=\lbrace f\in L^p(0,1)\mid f^\prime,f^{\prime\prime}\in L^p(0,1)\rbrace$). \\  

\textbf{Acknowledgment:} We would like to thank Enrique Zuazua for having kindly invited the two first authors of the paper in DeustoTech, Bilbao, Spain and for having pointed out the reference \cite{haraux1D}, which has been crucial for our analysis. We would like to thank Nicolas Burq for interesting discussions about the negative results of the semigroup associated to the D'Alembertian in dimension $n\geq 2$.

\section{Main results}
\label{sec_main}

In this paper, the class of nonlinearity that we will consider is defined as follows.
\begin{definition}[Nonlinear damping]
\label{def-damping}
A function $\sigma:\mathbb{R}\rightarrow \mathbb{R}$ is said to be a \emph{nonlinear damping function} if 
\begin{itemize}
\item[1.] it is locally Lipschitz and odd;
%\item[2.] one has $\sigma(0)=0$;
\item[2.] for any $s\in \mathbb{R}\setminus \lbrace 0\rbrace$, $\sigma(s)s>0$;
\item[3.] the function $\sigma$ is differentiable at $s=0$ with $\sigma'(0)=C_1$ for some $C_1>0$.
%\item[3.] the function $\sigma$ 
\end{itemize}
\end{definition}
Note that it follows from the above assumptions that $\sigma(0)=0$ and some of our subsequent results can also be derived without assuming that $\sigma$ is odd. Let us give now some examples of nonlinear dampings satisfying the properties given in Definition \ref{def-damping}. 
%Due to this definition (especially, item 2), the origin is an equilibrium point for \eqref{wave-nonmonotone}. Note that this nonlinearity is not assumed to be \emph{monotone}. For some of our results, such a property will be needed.
%\begin{definition}[Monotone damping]
%A scalar nonlinear damping is said to be monotone if for all $s_1,s_2\in\mathbb{R}$, one has
%\begin{equation}
%(\sigma(s_1)-\sigma(s_2))(s_1-s_2)\geq 0.
%\end{equation}
%\end{definition}
%
%The monotone property is in many cases really useful for either the well-posedness of the equation or the asymptotic stability of the origin. We refer the reader to \cite{marx2018stability} for a complete discussion on this topic. Finally, for some of our results, in some of our results, we will need also the scalar nonlinear damping to be linearly bounded. This property is stated in the following definition.
%
%\begin{definition}[Linearly bounded nonlinear damping]
%\label{def-sat}
%A scalar damping function is said to be a \emph{linearly bounded} if there exists a positive constant $m$ such that
%\begin{equation}
%|\sigma(s)|\leq m|s|,\quad \forall s\in \mathbb{R}.
%\end{equation}
%\end{definition}

\begin{example}[Example of scalar nonlinear dampings]
Below are listed some examples of nonlinear dampings:
\begin{itemize}
\item[1.] The classical saturation, defined as follows:
\begin{equation}
\sigma(s)=\mathrm{sat}(s):=\left\{
\begin{split}
&s  &\text{ if } |s|\leq 1,\\
&\frac{s}{|s|}  &\text{ if } |s|\geq 1,
\end{split}
\right.
\end{equation}
satisfies all the properties of Definition \ref{def-damping}.
\item[2.] The following nonlinearity
\begin{equation}
\label{example-nonmonotone}
\sigma(s) = \mathrm{sat}\left(\frac{1}{4}s-\frac{1}{30}\sin(10s)\right)
\end{equation}
is also a nonlinear damping. Note moreover that it is not monotone, as illustrated by Figure \ref{fig:nonmonotone}.
\item[3.] The dampings mentioned earlier in the introduction satisfies all the properties of Definition \ref{def-damping}. Recall that the first one is defined as follows: $\sigma$ is differentiable and satisfies, for some positive constants $k_0,k_1$ and $R\geq r>0$,
\begin{equation}
k_0|s|^r \leq \sigma^\prime(s)\leq k_1(|s|^r + |s|^R),\:\forall s\in\mathbb{R}.
\end{equation}
%The second one, taken from \cite{amadori2019decay}, is defined as follows
%\begin{equation}
%0<k_1\leq\sigma(s)\leq k_2,\:\forall x\in [0,1],\: k_1,k_2>0
%\end{equation}
\item[4.] The nonmonotone damping under consideration in \cite{MVC} satisfies also all the properties of Definition \ref{def-damping}. It is a continuously differentiable function defined as follows: 
\begin{equation}
\begin{split}
&s\sigma(s) \geq 0\\
&\sigma^\prime(s) \geq -m\\
c_1 \frac{|s|}{\left(\log(2+|s|)\right)^k}\leq\: &\sigma(s) \leq c_2 |s|^q,\: \forall s\text{ such that } |s|\geq 1
\end{split} 
\end{equation} 
with $c_1>0$, $c_2>0$, $m\geq 0$, $q\geq 0$ and $k\in [0,1]$.
%Recall that the estimations of the decay rate in \cite{MVC} apply for multi-dimensional wave equation. 
\end{itemize}
\end{example}

\begin{figure}[h!]
%\label{nonmonotone-example}
\begin{center}
\includegraphics[scale=0.6]{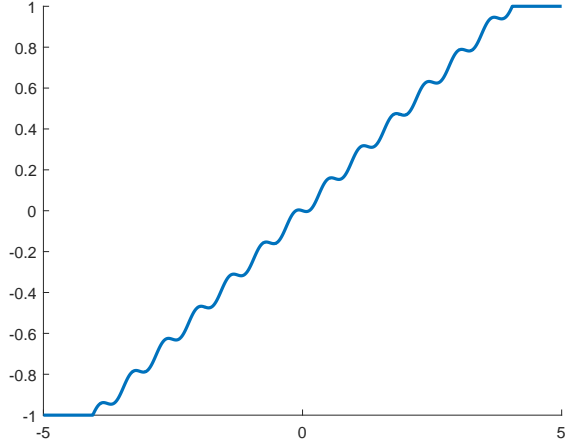}
\caption{For any $s\in [-5,5]$, the figure illustrates the function $\sigma$ given by \eqref{example-nonmonotone}}
\label{fig:nonmonotone}
\end{center}
\end{figure}

As illustrated in many papers \cite{enrike1990exponential}, \cite{alabau1999stabilisation}, \cite{marx2018stability}, some additional regularity properties are usually needed to obtain a characterization of the asymptotic stability of \eqref{wave-nonmonotone}. To be more precise, we need the state $z_t$ to be bounded in $L^\infty(0,1)$, for all $t\geq 0$. With a monotone nonlinearity $\sigma$, one would have this regularity result thanks to some nonlinear semigroup theorems. In the case where we consider nonmonotone nonlinearity, using the nonstandard functional space $H_\infty$, we can obtain such a regularity.

We are now in position to state the well-posedness results of our paper.
\begin{theorem}[Well-posedness]
\label{thm-wp}
\begin{itemize}
\item[1.] Assume that $\sigma$ is a \emph{nonlinear damping}. For any initial conditions $(z_0,z_1)\in H_\infty(0,1)$, there exists a unique solution $$z\in L^\infty(\mathbb{R}_+;W^{1,\infty}(0,1))\cap W^{1,\infty}(\mathbb{R}_+;L^\infty(0,1))$$ to \eqref{wave-nonmonotone}. Moreover, the following inequality is satisfied, for all $t\geq 0$
\begin{equation}
\Vert (z,z_t)\Vert_{H_\infty(0,1)}\leq 2\max(\Vert z_0^\prime\Vert_{L^\infty(0,1)}, \Vert z_1\Vert_{L^\infty(0,1)}).
\end{equation}
\item[2.] Assume that $\sigma$ is a \emph{nonlinear damping}, which is \emph{linearly bounded}, i.e., there exists $m>0$ such that
\begin{equation}
|\sigma(s)|\leq m|s|,\: s\in\mathbb{R}.
\end{equation}
For any initial conditions $(z_0,z_1)\in H_p(0,1)$, there exists a unique solution $z\in C(\mathbb{R}_+;W^{1,p}(0,1))\cap C^1(\mathbb{R}_+;L^p(0,1))$ to \eqref{wave-nonmonotone}. Moreover, the following inequality is satisfied, for all $t\geq 0$
\begin{equation}
\Vert (z,z_t)\Vert_{H_p(0,1)} \leq 2\Vert (z_0,z_1)\Vert_{H_p(0,1)}.
\end{equation}
\end{itemize}
\end{theorem}
Now that the functional setting is introduced, we are in position to state our asymptotic stability result. More precisely, the following results state that the trajectories of the system \eqref{wave-nonmonotone} converge \emph{exponentially} towards the equilibrium point $0$ and in some suitable norms, which will depend on the extra assumptions that we will impose.

%{\color{blue} il faut definir les notions de stabilite exponentielle et semi-global....}

\begin{theorem}[Semi-global exponential stability]
\label{thm-as}
Consider the 1D wave equation with Dirichlet boundary conditions and the nonlinear damping, as defined in
 \eqref{wave-nonmonotone}. Then, the following statements hold true. 
\begin{description}
\item[1.] For initial conditions $(z_0,z_1)\in H_\infty(0,1)$ satisfying:
\begin{equation}
\Vert (z_0,z_1)\Vert_{H_\infty(0,1)}\leq R,%\: \Vert z_1\Vert_{H_\infty(0,1)}\leq R,
\end{equation}
where $R$ is a positive constant, then, for any $p\in [2,\infty)$, there exist two positive constants $K:=K(R)$ and $\beta:=\beta(R)$ such that
\begin{equation}
\Vert (z,z_t)\Vert_{H_p(0,1)}\leq Ke^{-\beta t}\Vert (z_0,z_1)\Vert_{H_p(0,1)},\quad \forall t\geq 0.
\end{equation}
\item[2.] Consider initial conditions $(z_0,z_1)\in H_p(0,1)\cap D_2(0,1)$, $p\in (2,\infty)$, satisfying:
\begin{equation}
\Vert (z_0,z_1)\Vert_{D_2(0,1)} \leq R,
\end{equation}
where $R$ is a positive constant. Suppose moreover that $\sigma$ is \emph{monotone} and linearly bounded, i.e. there exists a positive constant $m$ such that:
\begin{equation}
|\sigma(s)|\leq m|s|,\: \forall s\in\mathbb{R}.
\end{equation}
Then, for any $q$ satisfying $2\leq q < p$, there exist two positive constants $K:=K(R)$ and $\beta:=\beta(R)$ such that
\begin{equation}
\Vert (z,z_t)\Vert_{H_q(0,1)}\leq Ke^{-\beta t}\Vert (z_0,z_1)\Vert_{H_q(0,1)},\quad \forall t\geq 0.
\end{equation}
\end{description}
\end{theorem}

The latter theorem is said to be a \emph{semi-global exponential stability} result, because the constants $K$ and $\beta$ depend on the bound of the initial condition, which is in contrast with the global case. It is neither local, because this bound can be arbitrarily large.

The previous theorem holds for the spaces $H_p$, but none of them provides an exponential stability result in the space $H_\infty$, which is an interesting functional space in practice. Indeed, an asymptotic stability result in this space allows to know how the amplitude of the states $(z,z_t)$ evolves. The following theorem allows one to obtain such a result at the price of assuming initial conditions in a more regular space than the ones considered in the previous theorem. 
%Once again, the following result states that the trajectories converge exponentially towards the equilibrium point $0$.
\begin{theorem}
\label{thm-infty}
Consider initial conditions $(z_0,z_1)\in D_\infty(0,1)$ satisfying:
\begin{equation}
z^\prime_0(x)= z^\prime_1(x)=0,\hbox{ and }\Vert (z_0,z_1)\Vert_{D_\infty(0,1)} \leq R,
\end{equation}
where $R$ is a positive constant. Suppose moreover that $\sigma$ and the localized function $a$ are \emph{continuously differentiable}. We assume also the derivative of $a$ is essentially bounded, i.e., there exists a positive constant $a_1$ such that
\begin{equation}
|a^\prime(x)|\leq a_1,\quad \hbox{a.e. } x\in [0,1].
\end{equation}
Finally, we assume that $\sigma$ is linearly bounded, i.e. there exists a positive constant $m$ such that
\begin{equation}
|\sigma(s)|\leq m|s|,\: \forall s\in\mathbb{R}.
\end{equation}
Then, there exist two positive constants, $K:=K(R)$ and $\beta(R)$ such that
\begin{equation}
\Vert (z,z_t)\Vert_{H_\infty(0,1)} \leq Ke^{-\beta t} \Vert (z_0,z_1)\Vert_{D_2(0,1)},\quad \forall t\geq 0.
\end{equation}
%{\color{blue} tres bizarre la norme $D_2$ et pas $D_\infty$}
\end{theorem}

\section{Exponential convergence result for a linear time-varying system}
\label{sec_appendix}

The proofs of the main results are mainly based on a result on a abstract linear time-varying system with a special structure. Since these results are used all along the paper, we provide it before the section devoted to the proofs of our main results. Note also that this exponential convergence result is of independent interest, in the sense that it can applied to more general systems than the nonlinear one given in \eqref{wave-nonmonotone}.  

\subsection{Preliminaries on linear time-varying systems}

Most of the results and definitions given in this section are borrowed from \cite{chicone-latushkin1999} and \cite{pazy1983semigroups}. A linear time-varying system on a Hilbert space\footnote{In the sequel, we will identify by $I_H$ the identity operator in the Hilbert space $H$} $H$ can be written as follows 
\begin{equation}
\left\{
\begin{split}
&\frac{d}{dt} z = A(t) z,\: &t\geq \tau\\
&z(\tau) = z_\tau,\: &\tau\geq 0,
\end{split}
\right.
\end{equation}
where $A:D(A(t))\subset H \rightarrow H$ is a (possibly unbounded) operator with $D(A(t))$ dense in $H$. For simplicity of the exposition, we will assume that the domain of the operator $D(A(t))$ is independent of $t$, i.e., if $D(A(t)):=\mathcal{D}\subset H$ where 
$\mathcal{D}$ is a dense subspace of $H$.
%In contrast with the case of autonomous abstract system on some Hilbert or Banach space where the trajectories are given by a one parameter semigroup (see e.g. \cite{miyadera1992nl_sg} for the case of nonlinear semigroups and \cite{tucsnak2009observation} for the linear semigroups, with, for the latter, a nice focus on abstract control systems), 
Trajectories of a linear time-varying system, are expressed by means of a \emph{two-parameter} family $(W(\theta,\tau))_{\theta\geq \tau}$ of bounded operators called an \emph{evolution family}. The presence of the second parameter, which does not appear in the case of strongly continuous semigroups, is due to the fact that the initialization time is crucial for time-varying systems. Indeed, as illustrated in \cite{chicone-latushkin1999}, asymptotic stability of such systems has to be proven uniformly with respect to $\tau$. Otherwise, different initialization times may lead to different asymptotic stability results.  

Let us define now evolution families
\begin{definition}[\cite{chicone-latushkin1999}, Definition 3.1., page 57]
\label{def-evolution-families}
A family of operators $(W(\theta,\tau))_{\theta\geq \tau}\subset \mathcal{L}(H)$, with $\theta,\tau\in\mathbb{R}$ or $\theta,\tau\in\mathbb{R}_+$, is said to be an evolution family if
\begin{itemize}
\item[(i)] $W(\theta,\tau)=W(\theta,s)W(s,\tau)$ and $W(\theta,\theta)=I_H$, for all $\theta\geq s\geq \tau$
\item[(ii)] for each $z\in H$, the function $(\theta,\tau)\mapsto W(\theta,\tau)z$ is continuous for $\theta\geq \tau.$
\end{itemize}
\end{definition}

In contrast with autonomous abstract system, general conditions to prove that a time-dependent and unbounded operator generates an evolution family are difficult to obtain (see e.g. the discussion in \cite[page 58]{chicone-latushkin1999}). However, if the domain of the operator $A(t)$ is independent of $t$, as it is the case here, one can derive some general conditions to prove that the operator $A(t)$ generates an evolution family, see e.g., \cite[Chapter 5]{pazy1983semigroups}, \cite{chicone-latushkin1999} or \cite{tanabe1960equations} for more informations on such conditions. Note that in Section \ref{sec_convergence}, we focus on a more structured abstract systems, closely related to the problem under consideration in this paper.

\subsection{Convergence result}
\label{sec_convergence}

This section is devoted to the statement and the proof of a theorem dealing with a time-varying linear infinite-dimensional system. It is inspired by \cite[Section 3.5]{liu1996finite}, which focuses on the case of finite-dimensional systems with a structure similar to those of the wave equation when defining it as an abstract system. As it is illustrated in Section \ref{sec_as} and in the proof of Item 1. of Theorem \ref{thm-wp}, we can transform \eqref{wave-nonmonotone} as a trajectory of a linear time-varying linear infinite-dimensional system. To define the related time-varying linear system, let us introduce two Hilbert spaces $H$ and $U$ equipped respectively with the norms $\Vert \cdot \Vert_H$ and $\Vert \cdot\Vert_U$, and the scalar products $\langle \cdot,\cdot\rangle_H$ and $\langle \cdot,\cdot\rangle_U$ respectively. The system under study in this section is the following:
\begin{equation}
\label{ltv-system}
\left\{
\begin{split}
&\frac{d}{dt} v = (A-d(t)BB^\star)v\\
&v(\tau)=v_\tau,
\end{split}
\right.
\end{equation}
where $A:D(A)\subset H \rightarrow H$, with $D(A)$ the domain of the operator $A$ that we suppose densely defined in $H$, $B\in\mathcal{L}(U,H)$ and $B^\star$ denotes the adjoint operator of $B$. We also assume that the restriction of the operator $BB^\star$  to $D(A)$, that we still denote $BB^\star$ by a slight abuse of notation, is a bounded operator in $D(A)$. Finally we suppose that $d:\mathbb{R}_+\to\mathbb{R}$ is essentially bounded, i.e.,  that there exists a positive constant $d_1$ such that $|d(t)|\leq d_1$, for a.e. $t\geq 0$. We assume that $A$ and $A^\star$ are dissipative and closed, then $A$ generates a strongly continuous semigroup of contractions, that we denote by $(e^{tA})_{t\geq 0}$. For $d(t)\in\mathbb{R}$, we define the unbounded operator
$$
A_{d(t)}:=A-d(t)BB^\star,
$$
whose domain is defined as $D(A_{d(t)}):=D(A)$. In particular, it is independent of $d$ (and thus of $t$). The graph norm of this operator is defined as follows
\begin{equation}
\Vert \cdot \Vert_{D(A_{d(t)})}:=\Vert \cdot \Vert_H + \Vert A_{d(t)} \cdot\Vert_H.
\end{equation}
%Since $BB^\star\in\mathcal{L}(H)$ and $d$ is essentially bounded, then there exist two positive constants $C_1$ and $C_2$ such that 
%\begin{equation}
%C_1 \Vert v\Vert_{D(A)} \leq \Vert v\Vert_{D(A_d)}\leq C_2 \Vert v\Vert_{D(A)},
%\end{equation}
%meaning that the two norms are equivalent.

We want to prove that the trajectories associated to the time-varying operator
$A_{d(t)}$ are well-defined, for any $T\geq \tau\geq 0$, in $C([\tau,T];H)$ for initial conditions in $H$, and in $C([\tau,T];D(A))$ for initial conditions in $D(A)$. By the use of the Duhamel's formula, trajectories of \eqref{ltv-system}, if they exist, they can be defined for $0\leq \tau\leq T$ as follows, for all $T\geq \tau$,
\begin{equation}
\label{Duhamel}
W(T,\tau)v_\tau:=e^{(T-\tau)A}v_\tau - \int_{\tau}^T e^{(T-\tau-s)} d(s)BB^\star v(s)ds.
\end{equation}
%where the solution is supposed to be $W(T,\tau)v_\tau$, for all $T\geq \tau$. 

Using the bound on the function $d$, one can show by a standard fixed-point argument together with the use of the Gronwall Lemma that the function $(T,\tau)\mapsto W(T,\tau)$ is well defined, and satisfies all the properties of an evolution family given in Definition \ref{def-evolution-families}. Solutions written as in \eqref{Duhamel} are usually called \emph{mild solutions} to \eqref{ltv-system}. Using the fact that the restriction of $BB^\star$ in $D(A)$ is bounded in $D(A)$, one can easily deduce that, if $v_\tau\in D(A)$, then, for a.e. $t\geq \tau$, one has $A_{d(t)}W(t,\tau)v_\tau\in D(A)$. This ensures in particular that there exists a unique strong solution to \eqref{ltv-system}. 

We are now in position to state the following result.
\begin{theorem}
\label{thm-ltv}
Consider the system given by \eqref{ltv-system}. Suppose that there exist positive constants $d_0$ and $d_1$ such that $d$ satisfies, for a.e. $t\geq 0$,
\begin{equation}
\label{bounds-D}
d_0 \leq d(t)\leq d_1,%\: \Vert S k(t)\Vert_{\mathcal{L}(U)}\leq d_2.  
\end{equation}
and assume that the origin of the following system
\begin{equation}
\label{lti-system}
\left\{
\begin{split}
&\frac{d}{dt} v = (A-d_0 BB^\star) v:=A_{d_0}v,\\
&v(0)=v_0
\end{split}
\right.
\end{equation}
is globally exponentially stable in $H$. Then, for any initial condition $v_0\in H$ initialized at $\tau=0$, the origin of \eqref{ltv-system} converges exponentially to $0$. 
\end{theorem}

\begin{proof}\textbf{ of Theorem \ref{thm-ltv}:}
Since the origin of \eqref{lti-system} is globally exponentially stable, then, due to \cite{datko1970extending}, there exist a self-adjoint operator $P\in\mathcal{L}(H)$ and a positive constant $C$ such that the following inequality holds true
\begin{equation}
\label{lyap-d0}
\langle PA_{d_0} v,v\rangle_H + \langle Pv,A_{d_0}v\rangle_H\leq -C\Vert v\Vert_H^2,\: \forall v\in D(A)
\end{equation}
Moreover, note that $A_{d_0}$ is also a dissipative operator, which means in particular that
\begin{equation}
\langle A_{d_0} v,v\rangle_H + \langle v,A_{d_0}v\rangle_H \leq 0,\: \forall v\in D(A).
\end{equation}

Now, consider the following candidate Lyapunov functional for \eqref{ltv-system}:
\begin{equation}
V(v):=\langle Pv,v\rangle_H + M\Vert v\Vert_H^2,
\end{equation}
where $M$ is a positive constant which has to be defined. By relying on a usual density argument, it is enough to prove exponential stability for strong solutions. The time derivative of $V$ along the solutions to \eqref{ltv-system} starting in $D(A)$ yields
\begin{equation}
\begin{split}
\frac{d}{dt} V(v)= &\langle (P+MI_H)A_{d(t)}v,v\rangle_H + \langle (P+MI_H)v,A_{d(t)} v\rangle_H\\
= & \langle PA_{d_0} v,v\rangle_H + \langle Pv,A_{d_0}v\rangle_H + M(\langle A_{d_0} v, v\rangle_H + \langle v,A_{d_0}v\rangle_H)\\
 & -  \langle (d(t) - d_0 I_H)BB^\star v, (P+MI_H)v\rangle_H -\langle (P+MI_H)v, (d(t)-d_0I_H)BB^\star v\rangle_H.\\
\leq & -C\Vert v\Vert_H^2 - d_0M\Vert B^\star v\Vert^2_{U} + 2(d_1-d_0)\Vert B^\star P\Vert_{\mathcal{L}(H,U)}\Vert B^\star v\Vert_{U} \Vert v\Vert_H \\
&-2M \langle (d(t)-d_0)B^\star v,B^\star v\rangle_U,
\end{split}
\end{equation}
where, in the third line, we have used the dissipativity of the operator $A$ and the Lyapunov inequality \eqref{lyap-d0}. 
 
Since $d(t)\geq d_0$ and $M>0$, one has $2M \langle (d(t)-d_0)B^\star v,B^\star v\rangle_U\geq 0$. By taking 
$$
M\geq \frac{2(d_1-d_0)^2\Vert B^\star P\Vert^2_{\mathcal{L}(H,U)}}{Cd_0},
$$
one obtains that 
$$
\frac{d}{dt} V(v) \leq - \frac{C}{2}\Vert v\Vert^2_H -\left(\sqrt{\frac{C}{2}} \Vert v\Vert_H^2 - \sqrt{\frac{2(d_1-d_0)^2\Vert B^\star P\Vert^2_{\mathcal{L}(H,U)}}{C}}\Vert B^\star v\Vert_U^2\right)^2,
$$ and therefore
%Moreover, applying Cauchy-Schwarz inequality, we known that there exists a positive number $\varepsilon$ such that
%\begin{equation}
%\begin{split}
%\frac{d}{dt} V(v) \leq &- C\Vert v\Vert_H^2 - d_0M\Vert B^\star v\Vert^2_U\\
%& + \frac{(d_1-d_0)\sqrt{\Vert BB^\star\Vert_{\mathcal{L}(H)}}\Vert P\Vert_{\mathcal{L}(H)}}{\varepsilon} \Vert v\Vert_H^2\\
%& + \varepsilon (d_1-d_0)\sqrt{\Vert BB^\star\Vert_{\mathcal{L}(H)}}\Vert P\Vert_{\mathcal{L}(H)} \Vert B^\star v\Vert_U^2
%\end{split}
%\end{equation}
%Setting
%\begin{equation}
%\varepsilon > \frac{(d_1-d_0)\Vert P\Vert_{\mathcal{L}(H)}\sqrt{\Vert BB^\star\Vert_{\mathcal{L}(H)}}}{C},
%\end{equation}
%and
%\begin{equation}
%M>\frac{\varepsilon(d_1-d_0)\Vert P\Vert_{\mathcal{L}(H)}\sqrt{\Vert BB^\star\Vert_{\mathcal{L}(H)}}}{d_0},
%\end{equation}
%one obtains
\begin{equation}
\frac{d}{dt} V(v) \leq -\frac{C}2\Vert v\Vert_H^2,\quad \forall v\in D(A).
\end{equation}
Since $V$ satisfies the following inequalities
\begin{equation}
M\Vert v\Vert^2_H\leq V(v) \leq (\Vert P\Vert_{\mathcal{L}(H)} + M)\Vert v\Vert^2_H
\end{equation}
one can conclude that 
\begin{equation}
\Vert v\Vert_H^2 \leq \frac{\Vert P\Vert_{\mathcal{L}(H)} + M}{M}\exp\left(-\frac{C}{2(\Vert P\Vert_{\mathcal{L}(H)}+M)}t\right) \Vert v_0\Vert_H^2,
\end{equation}
which ends the proof of Theorem \ref{thm-ltv}.
\end{proof}

\section{Proof of the main results}

\label{sec_proof}

\subsection{Proof of Theorem \ref{thm-wp}}

We now address the issue of well-posedness analysis of the wave equation \eqref{wave-nonmonotone} in the $L^p$ functional setting, for any $p\neq 2$. As mentionned in \cite{haraux1D,peral1980lp}, such a result does not hold for the wave equation with a dimension higher or equal to $2$ even with a monotone damping. 
%It is even more difficult to obtain such a result for a wave equation with a nonmonotone damping. %Indeed, as already illustrated in many papers such as \cite{map2017mcss}, \cite{enrike1990exponential}, \cite{alabau2012some} or \cite{slemrod1989mcss}, the monotone property allows to apply the LaSalle's Invariance Principle, which is not the case of PDEs with nonmonotone nonlinearity. Hence, proving only the asymptotic stability of systems with such nonlinearities is not an easy task. It is therefore clear that characterizing the asymptotic behavior of the associated trajectories is even more difficult.    

This section is divided into two parts. The first part is devoted to the proof of Item 1. of Theorem \ref{thm-wp}, while the second one proves Item 2. of Theorem \ref{thm-wp}. Recall that Item 2. of Theorem needs us to consider a quite general nonlinearity, while Item 2. is devoted to the case where this nonlinearity is linearly bounded.

\begin{proof}\textbf{ of Item 1. of Theorem \ref{thm-wp}:} 

The proof of Item 1. of Theorem \ref{thm-wp} relies on the following wave equation with a source term
\begin{equation}
\label{wave-source}
\left\{
\begin{split}
&z_{tt}(t,x) = z_{xx}(t,x) +h(t,x),\: (t,x)\in\mathbb{R}_+\times [0,1]\\
&z(t,0)=z(t,1) = 0,\: t\in\mathbb{R}_+\\
&z(0,x)= z_0(x),\: z_t(0,x) = z_1(x),\: x\in [0,1],
\end{split}
\right.
\end{equation}
where $h$ denotes the source term. From \cite[Theorem 1.3.8]{HarauxBook}, we know that, provided that $z_0,z_1\in H_2$ and that $h\in L^2(\mathbb{R}_+;L^2(0,1))$, there exists a unique solution $z\in C(\mathbb{R}_+;H_0^1(0,1))\cap C^1(\mathbb{R}_+;L^2(0,1))$ to \eqref{wave-source}. In particular, since $H_\infty(0,1)\subset H_2(0,1)$, this result holds true also for initial conditions $(z_0,z_1)\in H_{\infty}(0,1)$. The first step of our analysis in this section is to prove that picking initial conditions in $H_\infty(0,1)$ and the source term $h\in L^2(\mathbb{R}_+;L^\infty(0,1))$ improves also the regularity of the solution $z$ itself. 

To do so, our aim is to give an explicit formula for the latter equation, using the reflection method surveyed in \cite{strauss1992partial}. Roughly speaking, this method consists in extending the explicit formulation of trajectory of the wave equation in a bounded domain to the whole real line. We extend the initial conditions to be odd with respect to both $x=0$ and $x=1$, that is
\begin{equation}
\tilde{z}_0(-x) = -\tilde{z}_0(x)\text{ and } \tilde{z}_0(2-x)=-\tilde{z}_0(x),
\end{equation}
where $\tilde{z}_0$ denotes the $2$-periodic odd extension of $z_0$, i.e., 
\begin{equation}
\tilde{z}_0(x) = \left\{
\begin{split}
&z_0(x),\: 0<x<1\\
&-z_0(-x),\: -l<x<0\\
&\text{extended to be of period $2$.} 
\end{split}
\right.
\end{equation}
We can define similarly a $2$-periodic odd extension of $z_1$ (resp. $h$), denoted by $\tilde{z}_1$ (resp. $\tilde{h}$). Thanks to \cite[Theorem 1, Page 69]{strauss1992partial}, we can therefore define the explicit trajectory $z$ of \eqref{wave-source} (known as the D'Alembert formula) as follows:
\begin{equation}
\begin{split}
z(t,x) = &\frac{1}{2} \left[\tilde{z}_0(x+t)+\tilde{z}_0(x-t)\right] + \frac{1}{2}\int_{x-ct}^{x+ct}\tilde{z}_1(s) ds + \frac{1}{2}\int_0^t\int_{x-(t-s)}^{x+(t-s)} \tilde{h}(w,s)dwds. 
\end{split}
\end{equation}
We can further define $z_t$ as follows
\begin{equation}
\begin{split}
z_t(t,x) =& \frac{1}{2} (\tilde{z}_0^\prime(x+t)-\tilde{z}_0^\prime(x-t)) + \frac{1}{2}\left(\tilde{z}_1(x+t) - \tilde{z}_1(x-t)\right)\\
&+\frac{1}{2}\int_0^t \left(\tilde{h}(s,x+(t-s)) -\tilde{h}(s,x-(t-s))\right) ds 
\end{split}
\end{equation}
It is clear from these two latter equations that, when picking $(z_0,z_1)\in H_\infty(0,1)$ and $h\in L^2(\mathbb{R}_+;L^\infty(0,1))$, then $$z\in C(\mathbb{R}_+;W^{1,\infty}(0,1))\cap C^1(\mathbb{R}_+;L^\infty(0,1)).$$
We assume now that $\tilde{h}$ is written as follows
\begin{equation}
\label{source-term}
\tilde{h}(t,x):=-\tilde{a}(x)\sigma(y(t,x)),
\end{equation}
where $\tilde{a}$ is a $2$-periodic extension of $a$, $y\in L^2(0,T;L^2(0,1))$ and $\sigma$ is a scalar nonlinear damping (which is odd, due to Item 1 of Definition \ref{def-damping}). In particular, it means that $h\in L^2(\mathbb{R}_+;L^2(0,1))$. The proof of Theorem \ref{thm-wp} consists first in applying a fixed-point theorem, which will allow us to prove the well-posedness of \eqref{wave-nonmonotone} for a small time $T>0$ and second in using a stability result in \cite{haraux1D}, stated as follows.
\begin{proposition}[\cite{haraux1D}]
Let us consider initial condition $z_0,z_1\in H_\infty(0,1)$. If there exists a solution to \eqref{wave-nonmonotone}, therefore the time derivative of the following functional along the trajectories of \eqref{wave-nonmonotone}
\begin{equation}
\Phi(t):=\int_{0}^1 \left[F(z_t+z_x)+F(z_t-z_x)\right](t,x)dx,
\end{equation}
with $F$ any even and convex function, satisfies
\begin{equation*}
\frac{d}{dt}\Phi(t)\leq 0.
\end{equation*}
\end{proposition}
The latter proposition implies that
\begin{equation}
\Phi(t)\leq \Phi(0),\quad \forall t\geq 0.
\end{equation}
In particular, following the discussion in the proof of Corollary 2.3. in \cite{haraux1D}, if one picks $F(s)=[\mathrm{Pos}(|s|-2\max(\Vert z_0\Vert_{L^\infty(0,1)}, \Vert z_1\Vert_{L^\infty(0,1)}))]^2$, where the function $\mathrm{Pos}\: : s\in\mathbb{R}\rightarrow \mathrm{Pos}(s)\in\mathbb{R}_+$ is defined as follows:
\begin{equation*}
\mathrm{Pos}(s)=\left\{
\begin{split}
&s\text{ if } s>0\\
&0\text{ if } s\leq 0,
\end{split}
\right.
\end{equation*} then one obtains that

\begin{equation}
\Phi(z,z_t)=0,
\end{equation}
which implies that, for all $t\geq 0$
\begin{equation}
\max(\Vert z_x(t,\cdot)\Vert_{L^\infty(0,1)},\Vert z_t(t,\cdot)\Vert_{L^\infty(0,1)})\leq 2\max(\Vert z_0^\prime\Vert_{L^\infty(0,1)}, \Vert z_1\Vert_{L^\infty(0,1)}).
\end{equation}
Noticing that $$\Vert (z,z_t)\Vert_{H_\infty(0,1)}\leq \max(\Vert z_x(t,\cdot)\Vert_{L^\infty(0,1)},\Vert z_t(t,\cdot)\Vert_{L^\infty(0,1)}),$$ it is clear then that, for all $t\geq 0$
\begin{equation}
\label{estimate-Haraux}
\Vert (z,z_t)\Vert_{H_\infty(0,1)} \leq 2\max(\Vert z_0^\prime \Vert_{L^\infty(0,1)}, \Vert z_1\Vert_{L^\infty(0,1)})
\end{equation}
This estimate implies that, once one is able to prove that there exists a solution $(z,z_t)$ of \eqref{wave-nonmonotone} in $$L^\infty([0,T];W^{1,\infty}(0,1))\times L^\infty([0,T];L^\infty(0,1)),$$ for a small time $T>0$, then the well-posedness of \eqref{wave-nonmonotone} is ensured in $L^\infty(\mathbb{R}_+;W^{1,\infty}(0,1))\times L^\infty(\mathbb{R}_+;L^\infty(0,1))$.

%Now, we are in position to prove Item 1. of Theorem \ref{thm-wp}.

%\begin{proof}\textbf{ of Item 1. of Theorem \ref{thm-wp}:}

For $T>0$, let us define $\mathcal{F}_T$ the space of measurable functions defined on $[0,T]\times\mathbb{R}$ which are bounded, odd and $2$-periodic in space. We endow $\mathcal{F}_T$ with the $L^\infty$-norm so that it becomes a Banach space. Hence, denoting by $\Vert \cdot\Vert_{T}$ the norm of the latter functional space, we have, for every $y\in\mathcal{F}_T$
\begin{equation}
\Vert y\Vert_T:=\sup_{(t,x)\in [0,T]\times\mathbb{R}} |y(t,x)|.
\end{equation}
Let us consider $\mathbf{B}_K(y)$ the closed ball in $\mathcal{F}_T$ centered at $y\in\mathcal{F}_T$ of radius $K\geq 0$, where $K$ remains to be defined. %Note that it is a convex subset of $\mathbb{F}_T$. 
We define the mapping with which we will apply a fixed-point
\begin{equation}
\begin{split}
\Phi_T : \mathcal{F}_T &\rightarrow \mathcal{F}_T\\
y &\mapsto \Phi_T(y),
\end{split}
\end{equation}
where 
\begin{equation}
\begin{split}
&\Phi_T(y) = \frac{1}{2}\left(\tilde{z}^\prime_0(x+t) - \tilde{z}_0^\prime(x-t)\right) + \frac{1}{2}\left(\tilde{z}_1(x+t)-\tilde{z}_1(x-t)\right)\\
& -\frac{1}{2}\int_0^t \left(\tilde{a}(x+t-s)\sigma(y(s,x-(t-s))+ \tilde{a}(-x+t-s)\sigma(y(s,-x-(t-s))\right)ds.
\end{split}
\end{equation}
Using the fact that $\sigma$ is locally Lipschitz, note that, for every $y_0\in\mathcal{F}_T$ and $K>0$, there exists a positive constant $C(y_0,K)$ such that, for every $y,\tilde{y}\in \mathbf{B}_K(y_0)$ 
\begin{equation}
\label{contraction}
\Vert \Phi_T(y)-\Phi_T(\tilde{y})\Vert_T\leq C(y_0,K)T\Vert y-\tilde{y}\Vert_{T}.
\end{equation}
Choose $K:=2\left(\Vert \Phi_T(y_0)\Vert_T + \Vert y_0\Vert_T\right)$
and $T$ small enough such that
\begin{equation}
\label{choice-T}
C(y_0,K)T\leq 1.
\end{equation}
Consider a sequence $(y_n)_{n\in\mathbb{N}}$ defined as 
\begin{equation*}
y_{n+1} = \Phi_T(y_n).
\end{equation*}
%By an easy induction, one can prove that
%\begin{equation*}
%\Vert y_n-y_0\Vert_T \leq \left(\sum_{i=0}^n (C(y_0,K)T)^{i}\right) \Vert y_0-y_1\Vert_T
%\end{equation*}
%In particular, because of the choice of $T$, one can deduce from this inequality that
%\begin{equation}
%\Vert y_n\Vert_T\leq 2\Vert y_0\Vert_T+\Vert \Phi_T(y_0)\Vert_T\leq K.
%\end{equation}
%This implies that, for every $n\in\mathbb{N}$, one has $y_n\in\mathbf{B}_K(y_0)$. 

If this sequence converges to some $y^\star$, then \eqref{contraction} yields $y^\star$ is a fixed point for the mapping $\Phi_T$, which implies in particular that \eqref{wave-nonmonotone} is well-posed in the desired functional spaces. To prove the convergence of this sequence, it is enough to show that it is a Cauchy sequence. By induction, one can prove that
\begin{equation}
\label{inequality-contraction}
\Vert y_{n+1}-y_{n}\Vert_T\leq (C(y_0,K)T)^n \Vert y_0-y_1\Vert_T
\end{equation} 
and
\begin{equation}
\label{ball-property}
y_n\in \mathbf{B}_K(y_0).
\end{equation}
Indeed, thanks to the choice of $T$ in \eqref{choice-T}, these two properties are easily proved for $n=0$ and, moreover, we have with \eqref{contraction}, for all $n\geq 1$
\begin{equation}
\begin{split}
\Vert y_{n+1}-y_n\Vert_T=&\Vert \Phi(y_n)-\Phi(y_{n-1})\Vert_T\\
\leq &C(y_0,K)T\Vert \Phi(y_{n-1})-\Phi(y_{n-2})\Vert_T\\
\leq &(C(y_0,K)T)^{n}\Vert y_0-y_1\Vert_T.
\end{split} 
\end{equation}
The inequality \eqref{inequality-contraction} can be deduced from the above inequality. The property \eqref{ball-property} can be proved as follows:
\begin{equation}
\begin{split}
\Vert y_{n+1}\Vert \leq & \Vert y_n\Vert_T + \Vert y_1\Vert_T + \Vert y_0\Vert_T\\
\leq & \Vert y_0\Vert_T + \Vert \Phi_T(y_0)\Vert_T + \Vert \Phi_T(y_0)\Vert_T + \Vert y_0\Vert_T\leq K,\\
\end{split}
\end{equation}
where in the first line we have used \eqref{inequality-contraction} and, in the second line, we have used the fact that $y_n\in\mathbf{B}_K(y_0)$. 

The two properties \eqref{inequality-contraction} and \eqref{ball-property} show that the sequence $(y_n)_{n\in\mathbb{N}}$ is a Cauchy sequence. Since $\mathbf{B}_K(y_0)$ is a complete set, this sequence is therefore convergent. This means in particular that there exists a fixed-point to the mapping $\Phi_T(y_0)$, which implies that, for sufficiently small time $T$, there exists a unique solution $z\in L^\infty(0,T;W^{1,\infty}(0,1))$ and $z_t\in L^\infty(0,T;L^\infty(0,1))$. Thanks to \eqref{estimate-Haraux}, we can deduce that there exists unique solution $z\in L^\infty(\mathbb{R}_+;W^{1,\infty}(0,1))$ and $z_t\in L^\infty(\mathbb{R}_+;L^\infty(0,1))$. This concludes the proof of Item 1. of Theorem \ref{thm-wp}.
\end{proof}

\begin{proof}\textbf{ of Item 2. of Theorem \ref{thm-wp}:}

The proof of Item 2. of Theorem \ref{thm-wp} relies mainly on the following result.
\begin{proposition}[\cite{haraux1D}, Corollary 2.2.]
\label{corollary-Haraux}
Let $z\in C(\mathbb{R}_+,H^1_0(0,1))\cap C(\mathbb{R}_+,L^2(0,1))$ be a solution to
\begin{equation}
\label{ltv-wave-haraux}
\left\{
\begin{split}
&z_{tt}(t,x) - z_{xx}(t,x) + d(t,x)z_t =0,\: (t,x)\in\mathbb{R}_+\times [0,1]\\
&z(t,0) = z(t,1) = 0,\: t\in\mathbb{R}_+,
\end{split}
\right.
\end{equation}
with $d=d(t,x)\in L^\infty_{\mathrm{loc}}(\mathbb{R}_+,L^\infty(0,1)),\: a\geq 0 \text{ a.e. in }\mathbb{R}_+\times [0,1]$. Then, for any even, convex function $F\in C^1(\mathbb{R}$, the function
\begin{equation}
\Phi(t):=\int_0^1 [F(z_t+z_x)+F(z_t-z_x)](t,x) dx,
\end{equation} 
is non-increasing for all $t\geq 0$, whenever $\Phi(0)$ is finite. 
\end{proposition}
Consider now any trajectory $(z,z_t)$ of \eqref{wave-nonmonotone}. One can then see that such a trajectory is also a trajectory of the following linear time-varying system of the 
form \eqref{ltv-wave-haraux} given by
\begin{equation}
\label{ltv-wave}
\left\{
\begin{split}
&y_t(t,x) = y_{xx}(t,x) - d(t,x)y_t(t,x),\: (t,x)\in\mathbb{R}_+\times [0,1]\\
&y(t,0) = y(t,1) = 0,\: t\in\mathbb{R}_+\\
&y(0,x) = y_0(x), \: y_t(0,x) = y_1(x),\: x\in [0,1]
\end{split}
\right.
\end{equation}
with
\begin{equation}
\label{d-wave}
d(t,x) = \left\{
\begin{split}
&\frac{\sigma(z_t)}{z_t},\quad &\text{ if } z_t\neq 0,\\
&C_1,\quad &\text{ if } z_t = 0,
\end{split}
\right. 
\end{equation}
where one must take $y(0,\cdot) = z_0(\cdot)$ and $y_t(0,\cdot) = z_1(\cdot)$ to 
recover $(z,z_t)$ and where $C_1$ is the constant given in Item 3 of Definition \ref{def-damping}. Note that the function $d$ is clearly bounded in a neighborhood of $z_t=0$ since $\sigma$ is differentiable at $0$, as stated in Item 3 of Definition \ref{def-damping}. In order to apply Proposition \ref{corollary-Haraux}, one needs to check whether 
\begin{itemize}
\item[1.] $d(t,x)$ is nonnegative for a.e. $(t,x)\in\mathbb{R}_+\times [0,1]$, 
\item[2.] $d\in L^\infty_{\mathrm{loc}}(\mathbb{R}_+,L^\infty(0,1))$
\end{itemize}
Then, it is easy to see that the proof of Item 2. of Theorem \ref{thm-wp} follows by setting $F(s)=\frac{|s|^p}{p}$ and by considering initial conditions $(z_0,z_1)\in H_p(0,1)$, which implies in particular that $\Phi(0)$, given in Proposition \ref{corollary-Haraux}, is finite. Indeed, it is clear that there exists a positive constant $C$ such that
\begin{equation}
\Phi(0) \leq C \Vert (z_0,z_1)\Vert_{H_p(0,1)}. 
\end{equation}

Because of Item 2 of Definition \ref{def-damping}, one has that
\begin{equation}
d(t,x) \geq 0,\: \forall (t,x)\in\mathbb{R}_+\times [0,1]
\end{equation}
Since the damping under consideration is linearly bounded, then
\begin{equation}
d(t,x) \leq m.
\end{equation}
This implies in particular that $d\in L^\infty(\mathbb{R}_+,L^\infty(0,1)$. Hence, applying Proposition \ref{corollary-Haraux} allows one to conclude the proof of Item 2. of Theorem \ref{thm-wp}. 
\end{proof}

\subsection{Proof of Theorem \ref{thm-as}}
\label{sec_as}

The proof of Theorem \ref{thm-as} is divided into two steps: first, we associate with any trajectory of \eqref{wave-nonmonotone} a system of the form \eqref{ltv-wave-haraux}, so that the trajectory of \eqref{wave-nonmonotone} is  a trajectory of the associated linear time-varying system. One then applies Theorem \ref{thm-ltv} for the space $H_2(0,1)$, which is indeed the only Hilbert space among all the spaces $H_p(0,1)$. Second, using the fact that the solutions are bounded in $H_\infty(0,1)$ (resp. $H_p(0,1)$, with finite $p$) thanks to Theorem \ref{thm-wp}, and invoking the Riesz-Thorin theorem (see e.g., \cite[Theorem 1.1.1, Page 2]{bergh2012interpolation}), we conclude for each Item of Theorem \ref{thm-as}. 
%Let us state a simplified version of the Riesz-Thorin theorem, taken from \cite{bergh2012interpolation}, Theorem 1.1.1, Page 2.
%\begin{theorem}[\cite{bergh2012interpolation}, Theorem 1.1.1, Page 2]
%\label{thm-thorin}
%Let $p_0,p_1\in [1,\infty]$ such that $p_0<p_1$. Let $T:L^{p_0}(0,1)\rightarrow L^{p_0}(0,1)$ be a bounded \emph{linear operator} with norm 
%\begin{equation}
%M_{p_0} = \Vert T\Vert_{\mathcal{L}(L^{p_0}(0,1),L^{p_0}(0,1))}.%, M_{q}=\Vert T\Vert_{\mathcal{L}(L^q(0,1),L^q(0,1))}
%\end{equation}
%Suppose moreover that $T:L^{p_1}(0,1)\rightarrow L^{p_1}(0,1)$ is a bounded linear operator with norm
%\begin{equation}
%M_{p_1}:=\Vert T\Vert_{\mathcal{L}(L^{p_1}(0,1),L^{p_1}(0,1))}.
%\end{equation}
%Then, for all $p$ such that $p_0<p<p_1$, $T$ is a bounded operator from $L^{p}(0,1)$ into $L^p(0,1)$ and its norm satisfies
%\begin{equation}
%M_p\leq M_{p_0}^{1-\theta} M_{p_1}^{\theta},
%\end{equation}
%with $\theta$ given by the equation
%\begin{equation}
%\frac{1}{p} = \frac{1-\theta}{p_0} + \frac{\theta}{p_1}.
%\end{equation}
% \end{theorem}
%This statement is indeed simplified, because, in the statement given in \cite[Theorem 1.1.1., Page 2]{bergh2012interpolation}, the result applies to more general measures than the Lebesgue one, and to more general compact sets than the interval $[0,1]$. We made such a choice of presentation in order to make clearer the reading of our paper.

\begin{proof}\textbf{ of Item 1. of Theorem \ref{thm-as}:}

\textbf{First step: semi-global exponential stability in $H_2$.}
We first fix $p=2$, but still consider the initial conditions $(z_0,z_1)\in H_\infty(0,1)$. Moreover, given a positive constant, we consider that the initial conditions are such that:
\begin{equation}
\Vert (z_0,z_1)\Vert_{H_\infty(0,1)} \leq R.
\end{equation}
In particular, due to Theorem \ref{thm-wp}, one has
\begin{equation}
\Vert (z,z_t)\Vert_{H_\infty(0,1)} \leq 2R,\quad \forall t\geq 0.
\end{equation}
Given a solution $(z,z_t)$ of \eqref{wave-nonmonotone} with initial condition as above, one can see that it is a solution of the linear time-varying system given by 
\begin{equation}
\label{wave-ltv}
\left\{
\begin{split}
&y_{tt}(t,x) = y_{xx}(t,x)-d(t,x)a(x)y_t,\: (t,x)\in\mathbb{R}_+\times [0,1]\\
&y(t,0) = y(t,1) = 0,\: t\in\mathbb{R}_+\\
%&z(0,x) = z_0(x),\: z_t(0,x) = z_1(x),\: x\in [0,1],
\end{split}
\right.
\end{equation} 
with 
\begin{equation}
\label{d-wave}
d(t,x) = \left\{
\begin{split}
&\frac{\sigma(z_t(t,x))}{z_t(t,x)},\quad &\text{ if } z_t(t,x)\neq 0,\\
&C_1,\quad &\text{ if } z_t(t,x) = 0,
\end{split}
\right. 
\end{equation}
which is well defined. The solution $(z,z_t)$ of \eqref{wave-nonmonotone} is the trajectory of \eqref{d-wave} corresponding to the initial conditions $y(0,x) = z_0(x)$ and $y_t(0,x) = z_1(x)$.

The system given by \eqref{wave-ltv} is in the form \eqref{ltv-system}, with $H=H_2(0,1)$, $U=H_2(0,1)$, $D(A) = D_2(0,1)$, and the operators $A$ and $B$ defined as follows
\begin{equation}
\label{abstract-ltv-wave}
\begin{split}
A : D(A) \subset H &\rightarrow H\\
\begin{bmatrix}
v_1 & v_2
\end{bmatrix}^\top &\mapsto \begin{bmatrix}
v_2 & v_1^{\prime\prime}
\end{bmatrix}^\top
\end{split}
\end{equation}
and $B=\begin{bmatrix}0 & \sqrt{a(x)} \end{bmatrix}^\top$. It is easy to check that $A$ generates a strongly continuous semigroup of contractions (see \cite{HarauxBook}). Now, let us check whether $d(t,\cdot)$ satisfies \eqref{bounds-D}, for all $t\geq 0$.

Since the initial conditions $(z_0,z_1)\in H_\infty(0,1)$ and satisfy 
\begin{equation}
\Vert (z_0,z_1)\Vert_{H_\infty(0,1)} \leq R,
\end{equation}
for a given positive constant $R$, then invoking Theorem \ref{thm-wp} we have, for all $t\geq 0$
\begin{equation}
\sup_{x\in [0,1]} |z_t(t,x)|\leq 2R.
\end{equation} 
Moreover, since $d$ is a continuous function, there exist two positive constants $d_0$ and $d_1$ depending on $R$ such that 
\begin{equation}
d_0:=\min_{\xi\in [-2R,2 R]} \frac{\sigma(\xi)}{\xi} \leq d(t,x) \leq \max_{\xi\in [-2R,2R]} \frac{\sigma(\xi)}{\xi}:=d_1.
\end{equation}
Note moreover that the origin of the following system
\begin{equation}
\left\{
\begin{split}
&y_{tt}(t,x) = y_{xx}(t,x) -d_0a(x) y_t(t,x),\: (t,x)\in\mathbb{R}_+\times [0,1]\\
&y(t,0)=y(t,1)=0,\: t\in\mathbb{R}_+\\
&y(0,x)=y_0(x),\: y_t(0,x)=y_1(x),\: x\in [0,1].
\end{split}
\right.
\end{equation}
is exponentially stable in $H_2$ for any initial conditions $(y_0,y_1)\in H_2$ (see e.g. \cite{enrike1990exponential}) since $a$ is continuous and nonzero. The related operator of this system is $A-d_0BB^\star$, with domain $D(A_{d_0})=D_2$. Therefore, all the properties required in Theorem \ref{thm-ltv} are satisfied. Hence, there exist two positive constants $K:=K(R)$ and $\beta:=\beta(R)$ such that
\begin{equation}
\Vert W(t,0)(y_0,y_1)\Vert_{H_2} \leq K e^{-\beta t} \Vert (y_0,y_1)\Vert_{H_2},\quad \forall t\geq 0,
\end{equation}
where $(W(t,0))_{t\geq 0}$ is the evolution family associated to \eqref{ltv-system}, with the operators given in \eqref{abstract-ltv-wave}, starting from $v(0):=\begin{bmatrix}
y_0 & y_1
\end{bmatrix}^\top$. It corresponds to the trajectory of \eqref{wave-nonmonotone} starting from the particular initial conditions $(y_0,y_1)$ bounded by $R$ in the $H_\infty$-norm. 

\textbf{Second step: Semi-global exponential stability in $H_p(0,1)$}. Proceeding as in the proof of Item 1. of Theorem \ref{thm-as}, we will now use the notation 
\begin{equation}
\Theta(t):= W(t,0),
\end{equation}
which is a linear operator. From Theorem \ref{thm-wp}, we know that, for every initial condition $(z_0,z_1)\in H_\infty(0,1)$ satisfying $\Vert (z_0,z_1)\Vert_{H_\infty(0,1)}\leq R$, and noticing that the trajectory of \eqref{wave-nonmonotone} can be expressed with the evolution family $\Theta(t)$, one has
\begin{equation}
\Vert \Theta(t)(z_0,z_1)\Vert_{H_\infty(0,1)} \leq 2R, \quad \forall t\geq 0.
\end{equation}
Now, fix $t>0$. Note that $\Theta(t)$ is an operator from $(L^2(0,1))^2$ (resp. $(L^\infty(0,1))^2$) to $(L^2(0,1))^2$ (resp. $(L^\infty(0,1))^2$), if it associates $(z_0^\prime,z_1)\in L^2(0,1)^2$ (resp. $(z_0^\prime,z_1)\in L^\infty(0,1)^2$)  to $(z_x,z_t)\in L^2(0,1)$ (resp. $(z_x,z_t)\in L^\infty(0,1)^2$). Hence, we can apply the Riesz-Thorin theorem, with $p_0=2$, $p_1=\infty$ and $p_0<p<p_1$, and conclude that
\begin{equation}
\Vert \Theta(t)\Vert_{\mathcal{L}(L^p(0,1),L^p(0,1))}\leq (2R)^{\frac{p-2}{2}}\left(K\right)^{\frac{2}{p}}e^{-\frac{2\beta}{p}t}, \quad \forall t\geq 0,
\end{equation}
where $\Theta(t)$ corresponds exactly to the trajectory of \eqref{wave-nonmonotone} with the initial condition $(z_0,z_1)\in H_\infty(0,1)$. In particular, for every $(z_0,z_1)\in H_\infty(0,1)$, one therefore has, for every $t\geq 0$
\begin{equation*}
\begin{split}
\Vert \Theta (t) (z_0,z_1)\Vert_{H_p(0,1)} \leq &\Vert \Theta(t)\Vert_{\mathcal{L}(L^p(0,1),L^p(0,1))}\Vert (z_0,z_1)\Vert_{H_p(0,1)}\\
\leq & (2R)^{\frac{p-2}{2}}\left(\frac{K}{2}\right)^{\frac{2}{p}}e^{-\frac{2\beta}{p}t}\Vert (z_0,z_1)\Vert_{H_p(0,1)}.
\end{split}
\end{equation*}
This concludes the proof of Item 1. of Theorem \ref{thm-as}.\end{proof} 

\begin{proof}\textbf{ of Item 2. of Theorem \ref{thm-as}:}
\paragraph{First step: semi-global exponential stability in $H_2(0,1)$}
As mentioned in the introduction of this section, we first prove our results in $H_2(0,1)$. Indeed, this allows to apply some theoritical results from abstract control systems theory in Hilbert spaces.
%\footnote{Note that this results also apply to Banach spaces, but are more delicate to handle with, see e.g. \cite{miyadera1992nl_sg} for an overview} 
(see e.g. \cite{tucsnak2009observation} for a nice overview of such a theory). 

Hence, \eqref{wave-nonmonotone} can be rewritten in an abstract way
\begin{equation}
\label{abstract-wave}
\left\{
\begin{split}
&\frac{d}{dt} v= Av - BB^\star\tilde{\sigma}(v):=A_\sigma(v),\\
& v(0)=v_0,
\end{split}
\right.
\end{equation}
where $v=\begin{bmatrix}
v_1 & v_2
\end{bmatrix}^\top$, $A:v\in D_2(0,1)\subset H_2(0,1)\mapsto \begin{bmatrix}
0 & v_2\\
v_1^{\prime\prime} & 0
\end{bmatrix}\in H_2(0,1)$, $B:=\begin{bmatrix}
0 & \sqrt{a(x)} I_{H_2(0,1)}
\end{bmatrix}^\top$, $B^\star$ is the adjoint operator of $B$, and $\tilde{\sigma}$ might be defined as follows
\begin{equation}
\tilde{\sigma}(v):=\begin{bmatrix}
\sigma(v_1) \\ \sigma(v_2)
\end{bmatrix}.
\end{equation}

In this case, $\tilde{\sigma}$ is monotone and locally Lipschitz, then, invoving \cite[Lemma 2.1., Part IV, page 165]{showalter1997monobanach}, we know that $A_\sigma$ is a $m$-dissipative operator. In particular, $A_\sigma$ generates a strongly continuous semigroup of contractions, denoted by $(W_\sigma(t))_{t\geq 0}$. Due to \cite[Corollary 3.7., page 53]{miyadera1992nl_sg}, the following functions
\begin{equation}
t\mapsto \Vert W_\sigma(t)(v_0)\Vert_{H_2(0,1)},\: t\mapsto \Vert A_{\sigma}(W_\sigma(t)(v_0))\Vert_{H_2(0,1)}
\end{equation} 
are nonincreasing. This implies in particular that, for all $t\geq 0$
\begin{equation}
\Vert W_\sigma(t)(v_0)\Vert_{D_2(0,1)}\leq \Vert v_0\Vert_{D_2(0,1)}.
\end{equation}
From the Rellich-Kondrachov theorem \cite[Theorem 9.16, page 285]{brezis2010functional}, one has 
\begin{equation}
\Vert W_\sigma(t)(v_0)\Vert_{L^\infty(0,1)} \leq \Vert W_\sigma(t)(v_0)\Vert_{D_2(0,1)}. 
\end{equation}
This implies in particular that
\begin{equation}
\label{extra-regularity}
\Vert z_t(t,\cdot)\Vert_{L^\infty(0,1)} \leq \Vert (z_0,z_1)\Vert_{D_2(0,1)}\leq R. 
\end{equation}
Note that \eqref{abstract-wave} can be rewritten as a system in the form \eqref{ltv-system}, with $\tau=0$, setting
\begin{equation}
d(t,x) = \left\{
\begin{split}
& \frac{\sigma(z_t)}{z_t}\text{ if }& z_t\neq 0\\
& C_1 \text{ if } & z_t=0.
\end{split}
\right.
\end{equation}
Moreover, thanks to \eqref{extra-regularity}, and using the fact that $d$ is continuous, one has, for all $t\geq 0$
\begin{equation}
d_0:=\min_{\xi \in [-R,R]} \frac{\sigma(\xi)}{\xi}\leq d(t,x)\leq \max_{\xi\in [-R,R]} \frac{\sigma(\xi)}{\xi}:=d_1.
\end{equation}
It is also known moreover that the origin of
\begin{equation}
\left\{
\begin{split}
&z_{tt}(t,x)=z_{xx}(t,x) - d_0a(x) z_t(t,x),\: (t,x)\in\mathbb{R}_+\times [0,1]\\
&z(t,0)=z(t,1) = 0,\: t\in\mathbb{R}_+\\
&z(0,x) = z_0(x),\: z_t(0,x) = z_1(x),\: x\in [0,1]
\end{split}
\right.
\end{equation}
is globally exponentially stable (see e.g. \cite{enrike1990exponential}). This latter system can be written as an abstract system, with the unbounded operator $A-d_0BB^\star$. Hence, all the assumptions required by Theorem \ref{thm-ltv} are satisfied. This implies that there exist two positive constants $K:=K(R)$ and $\beta:=\beta(R)$ such that trajectories to \eqref{wave-nonmonotone} satisfy, for all $t\geq 0$
\begin{equation}
\label{1.semi-H2}
\Vert W_\sigma(t)(z_0,z_1)\Vert_{H_2(0,1)}\leq Ke^{-\beta t} \Vert (z_0,z_1)\Vert_{H_2(0,1)}.
\end{equation}

\paragraph{Second step: semi-global exponential stability in $H_p(0,1)$} Note that \eqref{abstract-wave} and \eqref{ltv-system} share a trajectory, i.e. the one in \eqref{ltv-system} corresponding to the initial condition initialized at $0$. The evolution family associated to this trajectory is given by
\begin{equation}
W(t,0):=\Theta(t),
\end{equation} 
where $(W(\theta,\tau))_{\theta\geq \tau}$ is the evolution family associated to \eqref{ltv-system}. This allows means that $W_\sigma(t)=\Theta(t)$, for all $t\geq 0$. Note that the operator $\Theta$ is linear, which is a crucial property in order to apply the Riesz-Thorin theorem. 
%It corresponds indeed to the operator $T$ in the statement of Theorem \ref{thm-thorin}. 

Because of Item 1. of Theorem \ref{thm-wp}, there exists a positive constant $C$ such that, for any initial condition $(z_0,z_1)\in H_p(0,1)$
\begin{equation}
\Vert \Theta(t) (z_0,z_1)\Vert_{H_p(0,1)} \leq C \Vert (z_0,z_1)\Vert_{H_p(0,1)}, \forall t\geq 0.
\end{equation}
Recall that 
$$
\Vert (z_0,z_1)\Vert_{H_p(0,1)} = \Vert z_0\Vert_{W_0^{1,p}(0,1)} + \Vert z_1\Vert_{L^p(0,1)}.
$$
First, since $\Vert z_0\Vert_{W_0^{1,p}(0,1)}=\Vert z_0^\prime\Vert_{L^p(0,1)}$, applying \cite[Theorem 8.8., page 212]{brezis2010functional} implies that there exists a positive constant $C$ such that
\begin{equation}
\begin{split}
\Vert z_0^\prime\Vert_{L^p(0,1)} \leq &\Vert z_0^\prime\Vert_{L^\infty(0,1)}
\leq  C\Vert z_0^\prime\Vert_{W^{1,2}(0,1)}
\leq  C\Vert z_0\Vert_{W^{2,2}(0,1)}.
\end{split}
\end{equation}
Similarly, one can also prove that
\begin{equation}
\Vert z_1\Vert_{L^p(0,1)} \leq \Vert z_1\Vert_{W^1_0(0,1)}.
\end{equation}
These two inequalities together with the definition of $H_p$ and those of $D_2$ imply that there exists a positive constant $C>0$ such that
\begin{equation}
\Vert (z_0,z_1)\Vert_{H_p(0,1)} \leq C\Vert (z_0,z_1)\Vert_{D_2}.
\end{equation}
Therefore, recalling that $\Vert (z_0,z_1)\Vert_{D_2}\leq R$, the operator $\Theta$ satisfies
\begin{equation}
\Vert \Theta(t)\Vert_{\mathcal{L}(L^p(0,1),L^p(0,1))}\leq CR,\: \forall t\geq 0.
\end{equation}
Moreover, since $\Theta(t) = W_\sigma(t)$, one has, using \eqref{1.semi-H2}
\begin{equation}
\Vert \Theta(t)\Vert_{\mathcal{L}(L^2(0,1),L^2(0,1))}\leq Ke^{-\beta t},\: \forall t\geq 0.
\end{equation} 
One can therefore apply the Riesz-Thorin theorem. This implies that, for any $q$ such that $2<q<p$
\begin{equation}
\Vert \Theta(t)\Vert_{\mathcal{L}(L^q(0,1),L^q(0,1))} \leq (CR)^{\frac{p(q-p)}{q(p-2)}}K^{\frac{2(p-q)}{q(p-2)}} e^{-\beta \frac{2(p-q)}{q(p-2)} t},\: \forall t\geq 0.
\end{equation}
It is clear then that
\begin{equation}
\Vert \Theta(t)(z_0,z_1)\Vert_{H_q(0,1)} \leq (CR)^{\frac{p(q-p)}{q(p-2)}}K^{\frac{2(p-q)}{q(p-2)}} e^{-\beta \frac{2(p-q)}{q(p-2)} t},\: \forall t\geq 0,
\end{equation}
which concludes the proof of Item 2. of Theorem \ref{thm-as}. 
\end{proof}

\subsection{Proof of Theorem \ref{thm-infty}}

The proof of this theorem consists in proving that the systems governed by the following states
\begin{equation}
u:=z_t,\: w:=z_x,
\end{equation} 
are semi-globally exponentially stable in the norm $H_2$. For example, for the state $u$, one would have
\begin{equation}
\Vert u \Vert_{H^1_0(0,1)} + \Vert u_t\Vert_{L^2(0,1)}\leq Ke^{-\beta t} \left(\Vert u_0 \Vert_{H^1_0(0,1)} + \Vert u_1\Vert_{L^2(0,1)}\right), \:\forall t\geq 0
\end{equation}
Then, invoking the Rellich-Kondrachov theorem, one can easily prove that the state $z_t$ converges to $0$ in the $L^\infty$-topology. To do so, we use the same strategy as in the proof of Theorem \ref{thm-as}. Indeed, as we will see it in the proof, the states $u$ and $w$ solve a linear time-varying system almost in the same form than the one introduced in Section \ref{sec_appendix}. We are now ready to give the proof of Theorem \ref{thm-infty}.

\begin{proof}\textbf{ of Theorem \ref{thm-infty}:}\\

The state $u$ solves the following PDE:
\begin{equation}
\label{system-u}
\left\{
\begin{split}
& u_{tt}(t,x) = u_{xx}(t,x) -a(x)\sigma^\prime(u(t,x)) u_t(t,x),\: (t,x)\in\mathbb{R}_+\times [0,1]\\
& u(t,0) = u(t,1) = 0,\: t\in\mathbb{R}_+\\
& u(0,x) = z_1(x):=u_0(x),\: u_t(0,x) = z_0^{\prime\prime}(x) - a(x)\sigma(z_1(x)):=u_1(x),\: x\in [0,1].
\end{split}
\right.
\end{equation}
We therefore end up with a trajectory of a linear time-varying system, as the one written in \eqref{ltv-system}, and with $d$ given by
\begin{equation}
d(t,x):=\sigma^\prime(u(t,x)),\: (t,x)\in\mathbb{R}_+\times [0,1].
\end{equation}
In order to apply Theorem \ref{thm-ltv}, we need to obtain upper and lower bounds for the function $d$.  Because $(z_0,z_1)\in D_\infty(0,1)$, the initial conditions are in particular in $H_\infty(0,1)$. Therefore, one can apply Theorem \ref{thm-wp}
\begin{equation}
\Vert (z,z_t)\Vert_{H_\infty(0,1)} \leq 2\Vert (z_0,z_1)\Vert_{H_\infty(0,1)}\leq 2R.
\end{equation}
Hence, because $d$ is continuous, and since its argument belongs to a compact, one can therefore prove that 
\begin{equation}
d_0:=\min_{\xi\in [-2R,2R]} \sigma^\prime(\xi)\leq d(t,x) \leq \max_{\xi\in [-2R,2R]} \sigma^\prime(\xi):=d_1.
\end{equation}
Moreover, we know that the origin of
\begin{equation}
\left\{
\begin{split}
&u_{tt}(t,x) = u_{xx}(t,x) - d_0 a(x) u_t(t,x),\: (t,x)\in\mathbb{R}_+\times [0,1],\\
&u(t,0)=u(t,1) = 0,\: t\in\mathbb{R}_+,\\
&u(0,x) = u_0(x),\: u_t(0,x) = u_1(x),\: x\in [0,1].
\end{split}
\right.
\end{equation}
is globally exponentially stable (see \cite{enrike1990exponential}). All the properties required in Theorem \ref{thm-ltv} are therefore satisfied. Applying it, we know that there exists two positive constants $K_1:=K_1(R)$ and $\beta_1:=\beta_1(R)$ such that
\begin{equation}
\Vert (u,u_t)\Vert_{H_2} \leq K_1 e^{-\beta_1 t} \Vert (u_0,u_1)\Vert_{H_2},\:\forall t\geq 0
\end{equation}
Recalling that $\Vert (u,u_t)\Vert_{H_2}=\Vert u_x\Vert_{L^2(0,1)}+\Vert u_t\Vert_{L^2(0,1)}$, and using the Rellich-Kondrachov theorem, one therefore ends up with
\begin{equation}
\Vert u(t,\cdot)\Vert_{L^\infty(0,1)} \leq K_1 e^{-\beta_1 t} \Vert (u_0,u_1)\Vert_{H_2(0,1)},\: \forall t\geq 0.
\end{equation}
Recalling the definition of $u$, one then obtains that
\begin{equation}
\label{ineq-infty-1}
\Vert z_t(t,\cdot)\Vert_{L^\infty(0,1)} \leq K_1 e^{-\beta_1 t} \Vert (z_0,z_1)\Vert_{D_2},\: \forall t\geq 0. 
\end{equation}
Let us now focus on the PDE solved by $w$, given by the following equations:
\begin{equation}
\left\{
\begin{split}
&w_{tt}(t,x) = w_{xx}(t,x) - a^{\prime}(x)\sigma(z_t(t,x)) - a(x) \sigma^\prime(z_t(t,x)) w_t(t,x),\: (t,x)\in\mathbb{R}_+\times [0,1],\\
&w(t,0)=w(t,1) = 0,\: t\in\mathbb{R}_+\\
&w(0,x) = z_0^\prime(x):=w_0(x),\: w_t(0,x) = z_1^\prime(x):=w_1(x),\: x\in [0,1].
\end{split}
\right.
\end{equation}
The boundary conditions $w(t,0)=w(t,1)=0$ are satisfied, because of the assumptions given in the statement of Theorem \ref{thm-infty}.

The system under consideration is slightly more complicated than the latter one. Indeed, it corresponds to a unhomogeneous linear time-varying system with \eqref{system-u} as corresponding homogeneous system. In an abstract framework, it can be rewritten as follows:
\begin{equation}
\label{inhomogeneous-ltv}
\left\{
\begin{split}
&\frac{d}{dt} v = Az-d(t)BB^\star v + F(t),\\
& v(0)=v_0,
\end{split}
\right.
\end{equation}
with $v=\begin{bmatrix}
w & w_t
\end{bmatrix}^\top$, $v_0=\begin{bmatrix}
w_0 & w_1
\end{bmatrix}^\top$ $A:D_2\subset H_2\rightarrow H_2$ defined as $A=\begin{bmatrix}
0 & I_{H_2}\\
\partial_{xx} & 0
\end{bmatrix}$, $B=\begin{bmatrix}
0 & \sqrt{a(x)}I_{H_2}
\end{bmatrix}^\top $, and $$
t\in\mathbb{R}_+\mapsto F(t):=-a^\prime(x)\sigma(z_t)\in\mathbb{R}_+.$$

From the discussion given in Section \ref{sec_appendix}, we know that the operator related to the homogeneous problem, which is in fact $A-d(t)BB^\star$, generates an evolution family, that we denote $(W(\theta,\tau))_{\theta\geq \tau}$. Since the homogeneous problem is the same than \eqref{system-u}, then the evolution family satisfies, for every $v\in H$
\begin{equation}
\label{estimationW}
\Vert W(t,0) v\Vert_H \leq K_1 e^{-\beta_1 t} \Vert v\Vert_H,\: t\geq 0.
\end{equation} 

Thanks to the Duhamel's formula, the trajectory of \eqref{inhomogeneous-ltv} can be rewritten as follows:
\begin{equation}
\label{Duhamel-w}
v(t) = W(t,0)v_0 + \int_0^t W(t-s,0)F(s) ds. 
\end{equation}

%The homogeneous version of the above mentioned system can be written
%\begin{equation}
%\left\{
%\begin{split}
%&w^h_{tt} = w^h_{xx} - a(x)\sigma^\prime(z_t) w^h_t,\\
%&w^h(t,0)=w^h(t,1) = 0,\\
%&w^h(0,x) = z_0^\prime(x):=w_0(x),\: w^h_t(0,x) = z_1^\prime(x):=w_1(x),
%\end{split}
%\right.
%\end{equation} 
%which is exactly the same system as the one defined in \eqref{system-u}. 

%{\color{blue} je ne comprend pas les ineggalites qui suivent. Il me semble que les calculs qui suivent justement sont une preuve pour les obtenir. Donc la phrase qui suit est prematuree, non??}

%Hence, one can apply the same strategy than before and obtain
%\begin{equation}
%\Vert W^h(t,0)(w_0,w_0)\Vert_{H_2(0,1)} \leq K_1 e^{-\beta_1 t} \Vert (w_0,w_1)\Vert_{H_2(0,1)},\:\forall t\geq 0,
%\end{equation}
%where $W^h$ corresponds to the evolution family of the homogeneous problem. 
%As before, one can also conclude that
%\begin{equation}
%\Vert z_x(t,\cdot)\Vert_{L^\infty(0,1)} \leq  K_1 e^{-\beta_1 t} \Vert (z_0,z_1)\Vert_{D_2(0,1)}
%\end{equation}

%From the discussion given in Section \ref{sec_appendix}, we know that the operator related to the homogeneous problem, which is in fact $A-d(t)BB^\star$, generates an evolution family, that we denote $(W(\theta,\tau))_{\theta\geq \tau}$. Then, thanks to the Duhamel's formula, the trajectory of \ref{inhomogeneous-ltv} can be rewritten as follows:
%\begin{equation}
%v(t) = W(t,0)v_0 + \int_0^t W(t-s,0)F(s) ds. 
%\end{equation}
Note that, since $\sigma$ is assumed to be linearly bounded, then for all $t\geq 0$
\begin{equation}
|F(t)|\leq m a_1 |z_t(t,x)|.
\end{equation}
Moreover, since $(z_0,z_1)\in D_\infty(0,1)$ and then, in particular, $(z_0,z_1)\in H_\infty(0,1)$, one can therefore apply Item 1. of Theorem \ref{thm-as}. Hence, considering initial condition $(z_0,z_1)\in H_\infty(0,1)$ satisfying $\Vert (z_0,z_1)\Vert_{H_\infty(0,1)}\leq R$, with $R$ positive, there exist two positive constants $K_2:=K_2(R)$ and $\beta_2:=\beta_2(R)$ such that
\begin{equation}
\Vert (z,z_t)\Vert_{H_2(0,1)} \leq K_2e^{-\beta_2 t} \Vert (z_0,z_1)\Vert_{H_2(0,1)}.
\end{equation}
Thanks to the latter inequality, the formula given in \eqref{Duhamel-w} and the fact that one can always assume $\beta_1<\beta_2$, one has, for all $t\geq 0$
\begin{equation}
\begin{split}
\Vert v(t)\Vert_{H_2} \leq & K_1e^{-\beta_1 t} \Vert v_0\Vert_{H_2(0,1)} + \int_0^t \Vert W(t-s,0)\Vert_{\mathcal{L}(H_2(0,1),H_2(0,1))}\Vert F(s)\Vert_{H_2} ds \\
\leq & K_1e^{-\beta_1 t} \Vert v_0\Vert_{H_2(0,1)} + K_2 m a_1 \Vert (z_0,z_1)\Vert_{H_2(0,1)}\int_0^t e^{-\beta_1(t-s)}e^{-\beta_2 s}ds\\
\leq & \left(K_1+\frac{K_2}{\beta_2-\beta_1}m a_1\right)e^{-\beta_1 t}\Vert v_0\Vert_{H_2(0,1)}
\end{split}
\end{equation}
where, in the second line, we have used \eqref{estimationW}. Recalling the definition of $v$ and those of $w$, one has therefore, for all $t\geq 0$
\begin{equation}
\Vert z_x\Vert_{H^1_0(0,1)} \leq \left(K_1+\frac{K_2}{\beta_2-\beta_1}m a_1\right) e^{-\beta_1 t}\Vert (z_0,z_1)\Vert_{D_2(0,1)}.
\end{equation} 
Thanks to the Rellich-Kondrachov theorem, we thus have, for all $t\geq 0$,
\begin{equation}
\label{ineq-infty-2}
\Vert z_x\Vert_{L^\infty(0,1)} \leq \left(K_1+\frac{K_2}{\beta_2-\beta_1}m a_1\right)e^{-\beta_1 t}\Vert (z_0,z_1)\Vert_{D_2(0,1)}.
\end{equation}
Due to the definition of the norm of $H_\infty(0,1)$ and gathering \eqref{ineq-infty-1} and \eqref{ineq-infty-2} together, one therefore obtains that
\begin{equation}
\Vert (z,z_t)\Vert_{H_\infty(0,1)}\leq \left(2K_1+\frac{K_2}{\beta_2-\beta_1}m a_1\right)e^{-\beta_1 t}\Vert (z_0,z_1)\Vert_{D_2(0,1)}
\end{equation}
This concludes the proof of Theorem \ref{thm-infty}. 
\end{proof}

\section{Conclusion}

\label{sec_conclusion}

%{\color{blue} a revoir stp. A mon avis on garde juste l'ouverture vers d'autres edp...}

In this paper, we have provided well-posedness and asymptotic stability results for a linear wave equation subject to a nonlinearity in $L^p$ functional spaces with $p\in [2,\infty]$. For future work, it would be interesting to obtain similar results (namely exponential stability with decay rate estimates) in $H_p$, $p\in(2,\infty]$, with  initial conditions in $H_p$
and also to other one-dimensional PDEs, such as the Korteweg-de Vries equation (either linearized or nonlinear). Finally, considering the case where a one-dimensional wave equation is controlled from the boundary, as in \cite{prieur2016wavecone}, could be also interesting, especially because, at the best of our knowledge, there is no $L^p$ analysis done for such an equation. Another extension would be the case of more complex hyperbolic systems, as the ones described in \cite{bastin2016stability} in the case of linear feedback laws. The $L^p$ framework could be also interesting in this case.

%\appendix

\bibliographystyle{plain}
\bibliography{bibsm}
\end{document}